\documentclass[11pt]{article}
\usepackage{amsmath}
\usepackage{amssymb}
\newcommand{\R}{{\mathbb R}}
\newcommand{\Rp}{{\mathbb R^+}}

\newcommand{\const}{{\rm Const}}

\newcommand{\dist}{\mbox{\rm dist}}

\newcommand{\N}{{\mathbb N}}

\newcommand{\LL}{{\mathbb L}}

\newcommand{\arrowd}{\mathop{\longrightarrow}_{\cal D}}

\newcommand{\arrowp}{\mathop{\longrightarrow}_{\cal P}}

\newcommand{\Rd}{{{{\mathbb R}^d}}}

\newtheorem{theorem}{\bf Theorem}[subsection]
\newtheorem{remark}[theorem]{\bf Remark}

\newenvironment{proof}{{\sc Proof.}}{\hfill $\Box$}

\newcommand{\nsubsection}{\setcounter{equation}{0}\subsection}

\begin{document}

\title {Weak and strong approximations of reflected diffusions
\\ via penalization methods}
\author { {\small Leszek S\l omi\'nski}
\\{\small  Faculty of Mathematics and Computer Science,
Nicolaus Copernicus University}\\
{\small ul. Chopina 12/18, 87--100 Toru\'n, Poland}}
\date{}
\maketitle
\begin{abstract}
We study approximations of reflected It\^o diffusions on convex
subsets $D$ of $\Rd$ by solutions of stochastic differential
equations with penalization terms. We assume that the diffusion
coefficients are merely measurable (possibly discontinuous)
functions. In the case of Lipschitz continuous coefficients we
give the rate of ${\LL}^p$ approximation for every $p\geq1$. We
prove that if $D$ is a convex polyhedron then the rate is ${\cal
O}\big((\frac{\ln n}n)^{1/2}\big)$, and in the general case the
rate is ${\cal O}\big((\frac{\ln n}n)^{1/4}\big)$.
\end{abstract}

\footnotetext{Research supported by Polish Ministry of Science and
Higher Education Grant N N201 372 436.}

\nsubsection{Introduction} In the paper we study weak and strong
approximations of solutions of  $d$--dimensional stochastic
differential equations (SDEs)
\begin{equation}
\label{eq1.1} X_t = x_0 + \int_0^t \sigma(s,X_s)\,dW_s+\int_0^t
b(s,X_s)\,ds+K_t, \quad t\in{\Bbb R}^+
\end{equation}
 with reflecting boundary condition
on a convex domain $D$. Here $x_0\in\bar D=D\cup\partial D$, $X$
is a reflecting process on $\bar D$,  $K$ is a bounded variation
process with variation $|K|$  increasing only, when
$X_t\in\partial D$, $W$ is a $ d$-dimensional standard Wiener
process  and $\sigma:{\Bbb R}^+\times\Rd\rightarrow{\Bbb
R}^d\otimes{\Bbb R}^d$, $b:{\Bbb R}^+\times\Rd\rightarrow{\Bbb
R}^d$ are measurable (possibly discontinuous) functions. Suppose
that for $n\in\Bbb N$ we are given measurable coefficients
$\sigma_n:{\Bbb R}^+\times\Rd\rightarrow{\Bbb R}^d\otimes{\Bbb
R}^d$, $b_n:{\Bbb R}^+\times\Rd\rightarrow{\Bbb R}^d$ and a
standard Wiener process $W^n$, and assume that there exists a
solution $X^n$ of the following SDE with penalization term
\begin{equation}
\label{eq1.2} X^n_t = x_0 + \int_0^t
\sigma_n(s,X^n_s)\,dW^n_s+\int_0^t b_n(s,X^n_s)\,ds
-n\int_0^t(X^n_s-\Pi(X^n_s))ds,\quad t\in{\Bbb R}^+,
\end{equation}
where  $\Pi(x)$  is the projection of $x$  on $\bar D$. The
problem is to find conditions on $\{\sigma_n\},\,\{b_n\}$ ensuring
convergence of $\{X^n\}$ to the reflected diffusion  $X$,  and
secondly, to give the rate of such convergence.

 Reflected diffusions  have many   applications, for instance in queueing
 systems, seismic reliability analysis and  finance (see e.g.
  Asmussen \cite{as}, Dupuis and Ramanan \cite{dr},  Kr\'ee and Soize \cite{KS}, Pettersson \cite{p1}, Shepp and Shiryaev \cite{ss}).
Therefore, the problem of practical approximations of  solutions
of (\ref{eq1.1}) is very important. Discrete penalization schemes
based on the approximation of $X$ by  solutions of equations with
penalization term are well known (see e.g. Pettersson \cite{p2},
Kanagawa and Saisho \cite{ks}, Liu \cite{li}, S\l omi\'nski
\cite{s4}).

Approximation of reflected diffusions via penalization methods was
earlier considered by Menaldi \cite{me}, Menaldi and Robin
\cite{mr}, Lions and Sznitman \cite {ls}, Lions, Menaldi and
Sznitman \cite{lms}, Storm \cite{st}, Saisho and  Tanaka
\cite{st} and many others. Unfortunately, these authors have
restricted themselves to the case of Lipschitz continuous
coefficients.

 In the present paper we consider  measurable coefficients  $\sigma_n,b_n$  such
 that
\begin{equation}
\label{eq1.3} \|\sigma_n(t,x)\|^2+|b_n(t,x)|^2\le
C(1+|x|^2),\quad(t,x)\in{\Bbb R}^+\times\Rd,\quad n\in\N
\end{equation}
for some  $C>0$.
 To prove convergence of $\{X^n\}$  to $X$ we first show that under (\ref{eq1.3})  the sequence
$\{X^n\}$ is very close to the sequence   $\{\Pi(X^n)\}$ and we
observe that    $\Pi(X^n)$ is a solution of some Skorokhod problem
(for the definition of the Skorokhod problem see Section 2). Next,
using a well developed theory of convergence of solutions of the
Skorokhod problem  (see e.g. \cite{as,rs1,s2,s4,ta}) we prove our
main approximation results. Moreover,  we are able to strengthen
the rate of the convergence of the penalization method in the
classical case of Lipschitz continuous coefficients $\sigma, b$.

The paper is organized as follows.

In Section 2   we estimate  the $\LL^p$ distance between $X^n$ and
$\bar D$.  Using some new estimates of $\LL^p$-modulus of
continuity of It\^o's processes from Fischer and Nappo \cite{fn}
we prove that under (\ref{eq1.3})  for every $p\geq1$, $T>0$,
\[||\sup_{t\leq T}
\dist(X^n_t,\bar D)||_p={\cal O}\big((\frac{\ln n}n)^{1/2}\big),\]
where $||\cdot||_p=(E(\cdot)^p)^{1/p}$ denotes the usual $\LL^p$
norm. We also show that $\{X^n\}$ is tight in $C({\Bbb R}^+,{\Bbb
R}^d)$ and its weak  limit point solve the Skorokhod problem.

Section 3 contains  our main results concerning weak and strong
approximations of solutions of  (\ref{eq1.1}). We consider the set
of conditions on coefficients from the  paper by Rozkosz and S\l
omi\'nski \cite{rs1} on stability of solutions of stochastic
differential equations with reflecting boundary. Roughly speaking,
we assume that $\{\sigma_n\}$, $\{b_n\}$ satisfy (\ref{eq1.3}) and
$\{(\det \sigma_n\sigma^{\star}_n)^{-1}\}$ is locally uniformly
integrable on some set ($\sigma^*_n$ denotes the matrix adjoint to
$\sigma_n$). Then we show that if $\{\sigma_n\},\,\{b_n\}$ tend to
$\sigma,\,b$ a.e. on the set mentioned above   and uniformly on
its completion, then $X^n\arrowd X$,  where $X$ denotes a unique
weak  solution of $(\ref{eq1.1})$. Under the additional
assumptions that $W^n\arrowp W$ and that (\ref{eq1.1}) is pathwise
unique we show that $X^n\arrowp X$. Thus, we generalize earlier
approximation results to equations with possibly discontinuous and
nonelliptic diffusion coefficients and discontinuous drift
coefficients.

Section 4  is devoted to the classical case, where all
coefficients are fixed Lipschitz continuous functions with respect
to $x$  and all stochastic integrals are driven by the same Wiener
process, i.e. $\sigma_n=\sigma$, $b_n=b$, and $W_n=W$, $n\in\N$,
and there is $L>0$ such that
\begin{equation}\label{eq1.4}
\|\sigma(t,x)-\sigma(t,y)\|^2+|b(t,x)-b(t,y)|^2\leq
L|x-y|^2,\quad(t,x)\in{\Bbb R}^+\times\Rd.
\end{equation}
In this case,  if $D$ is a convex polyhedron, we prove that for
every $p\geq1$, $T>0$, \[||\sup_{t\leq T}|X^n_t-X_t|||_p={\cal
O}\big((\frac{\ln n}n)^{1/2}\big).\]  For arbitrary convex domain
we prove that for every $p\geq1$, $T>0$,
\[||\sup_{t\leq T}|X^n_t-X_t|||_p={\cal O}\big((\frac{\ln
n}n)^{1/4}\big).\] Thus, we strengthen earlier results on the
subject proved by Menaldi \cite{me}.

In the sequel we use the following notation. ${\Bbb
R}^+=[0,\infty)$, $C({\Bbb R}^+,{\Bbb R}^d)$ is the space of
continuous functions  $x:{\Bbb R}^+\rightarrow {\Bbb R}^d$
equipped with the topology of uniform convergence on compact
subsets of ${\Bbb R}^+$. For every $x\in C({\Bbb R}^+,{\Bbb
R}^d)$, $\delta>0$, $T>0$ we set
$\omega_{\delta}(x,T)=\sup\{|x_t-x_s|;\,s,t\in[0,T],\,|s-t|\leq
\delta\}$.   ${\Bbb R}^d\otimes{\Bbb R}^d$ is the set of $(d\times
d)$\,-\,matrices.
 The abbreviation $a.e.$  means  ``almost everywhere"  with
respect to the Lebesgue measure, "$\arrowd$", "$\arrowp$" denote
convergence in law and in probability, respectively.

\nsubsection{General results} \label{sec2}

Let  $D$  be a nonempty convex domain in ${\Bbb R}^d$ and let
${\cal N}_x$ denote the set of inward normal unit vectors at
$x\in\partial D$. Note that ${\bf n}\in{\cal N}_{x} $ if and only
if $<y-x,\mbox{\bf n}>\geq 0 $  for every  $y\in\bar D$ (see e.g.
\cite{me,st}). Moreover, if $\dist(x,\bar D)>0$, then
\[
\frac{\Pi(x)-x}{|\Pi(x)-x|}\in{\cal N}_{\Pi(x)}.\] Let $Y$ be an
$\{{\cal F}_t\}$\,-\,adapted process with continuous trajectories.
We will say that a pair $(X,\,K)$  of $\{{\cal F}_t\}$\,-\,adapted
processes is a solution of the Skorokhod problem associated with
$Y$  if
\begin{eqnarray*}
& &X=Y+K,\\
& &X\mbox{ is $\bar D$\,-\,valued},\\
\label{eq8} & &K\mbox{ is a process with locally bounded variation
such
that $K_0=0$ and}  \\
& &\quad K_t=\int_0^t{\bf n}_s\,d|K|_s,\qquad |K|_t=\int_0^t{\bf
1}_{\{X_s\in\partial D\}}\,d|K|_s,
\quad t\in{\Bbb R}^+,\nonumber\\
& &\mbox{where }{\bf n}_s\in{\cal N}_{X_s}\mbox{ if
$X_s\in\partial D$}.
\end{eqnarray*}
It is well known that for every process $Y$  with  continuous
trajectories there exists a unique solution $(X,K)$ of the
Skorokhod problem associated with $Y$ (see e.g. \cite{ce} or
\cite{lau}, where a more general case of c\`adl\`ag processes is
considered). The theory of convergence of  solutions of the
Skorokhod problem is well developed (see e.g.
\cite{as,rs1,s2,s4,ta}). Unfortunately,  solutions of
(\ref{eq1.2}) are not solutions of the Skorokhod problem and the
problem of their convergence is more delicate.

Suppose that  we are given a filtered probability space
$(\Omega^n,\,{\cal F}^n,\,\{{\cal F}^n_t\},\,P^n)$ satisfying the
usual conditions  and a $d$\,-\,dimensional $\{{\cal
F}^n_t\}$\,-\,Wiener process $W^n$, $n\in\Bbb N$. Let $\{X^n\}$
denote the sequence of solutions of (\ref{eq1.2}). In the present
paper we will use the simple fact that under (\ref{eq1.3}) there
exists a sequence of solutions of the Skorokhod problem very close
to the sequence $\{X^n\}$. Observe that we can rewrite
(\ref{eq1.2}) into the form
\[\Pi(X^n_t)=Y^n_t-n\int_0^t(X^n_s-\Pi(X^n_s))ds,\quad t\in\Rp,\]
where $Y^n_t=x_0-X^n_t+\Pi(X^n_t)+\int_0^t
\sigma_n(s,X^n_s)\,dW^n_s+\int_0^t b_n(s,X^n_s)\,ds$, $t\in\Rp$,
$n\in\N$. Since $\Pi(X^n)\in\bar D$, $|K^n|$  increases only when
$\Pi(X^n)_t\in\partial D$ and
\[K^n_t=n\int_0^t\frac{\Pi(X^n_s)-X^n_s}{|\Pi(X^n_s)-X^n_s|}|\Pi(X^n_s)-X^n_s|ds=\int_0^t{\bf n}_s\,d|K^n|_s,\quad t\in\Rp,\]
it is clear that $(\Pi(X^n), K^n)$  is a solution of the Skorokhod
problem associated with $Y^n$, $ n\in\N$. One can also observe
that
\[|X^n_t-\Pi(X^n_t)|=\dist(X^n_t,\bar D),\quad t\in\Rp,\,n\in\N.\]
\begin{theorem} \label{thm1} Assume that (\ref{eq1.3}) is satisfied. \begin{enumerate}
 \item[{\bf (i)}] For every $p\geq1$,
$T>0$ there is $C>0$ such that \[||\sup_{t\leq T}\dist(X^n_t,\bar
D )|||_p\leq C(\frac{\ln n}{n})^{1/2},\quad n\in\N.\]
\item[{\bf (ii)}]
 $\{(X^n,K^n)\}_{n\in\Bbb N}$  is tight in $C({\Bbb R}^+,{\Bbb
R}^{2d})$ and its every weak limit point $(X,K)$ is a solution of
the Skorokhod problem. \end{enumerate}
\end{theorem}
\begin{proof} (i) Fix $T>0$. First observe that by \cite[Corollary 2.4]{laus} and
Gronwall's lemma,
\begin{equation}\label{eq2.1}
\sup_nE\sup_{t\leq T}|X^n_t|^p<+\infty
\end{equation}
for every $p\geq 1$. By the above  and estimates from Fischer and
Nappo \cite[Theorem 1]{fn} (see also \cite[Lemma 4.4]{p1} and
\cite[Lemma A4]{s4}) for every $p\geq1$
 there is  $C>0$ such that
\begin{equation}\label{eq2.2}
||\omega_{1/n}(\bar Y^n,T)||_p\leq C(\frac{\ln n}{n})^{1/2},\quad
n\in\N,\end{equation} where $\bar Y^n_t=x_0+\int_0^t
\sigma_n(s,X^n_s)\,dW^n_s+\int_0^t b_n(s,X^n_s)\,ds$, $t\in\Rp$,
$n\in\N$.

Fix $n\in N$  and $k=0,1,....,[nT]-1$. Clearly,  $X^n$  is a
solution of the equation
\begin{equation}\label{eq2.3} X^n_{k/n+s}=X^n_{k/n}+\bar Y^n_{k/n+s}-\bar
Y^n_{k/n}-n\int_0^s(X^n_{k/n+u}-\Pi(X^n_{k/n+u})du,\quad
s\in[0,1/n]\end{equation} on the interval $[k/n, (k+1)/n]$. It is
also clear that there exists a unique solution of the equation
\begin{equation}\label{eq2.4} \bar X^n_s=X^n_{k/n}-n\int_0^s(\bar
X^n_u-\Pi(\bar X^n_u))du,\quad s\in[0,1/n].\end{equation} One can
easily check  that $\bar
X^n_s=\Pi(X^n_{k/n})+(X^n_{k/n}-\Pi(X^n_{k/n})e^{-ns}$,
$s\in[0,1/n]$, which implies that
\begin{equation}\label{eq2.5}|\bar X^n_{1/n}-\Pi(\bar
X^n_{1/n})|\leq|\bar
X^n_{1/n}-\Pi(X^n_{k/n})|=|X^n_{k/n}-\Pi(X^n_{k/n})|e^{-1}.\end{equation}
Subtracting (\ref{eq2.4}) from (\ref{eq2.3})  we see that
\[X^n_{k/n+s}-\bar X^n_s=\bar Y^n_{k/n+s}-\bar Y^n_{k/n}-n\int_0^s(X^n_{k/n+u}-\bar X^n_u-\Pi(X^n_{k/n+u}+\Pi(\bar
X^n_n))du,\quad s\in[0,1/n], \] hence that
\[|X^n_{k/n+s}-\bar X^n_s|\leq \omega_{1/n}(\bar Y^n,T)+2n\int_0^s|X^n_{k/n+u}-\bar X^n_u|du, \quad
s\in[0,1/n],\] because $\Pi:\Rd\to\bar D$ is Lipschitz continuous
with the constant equal to $1$. Applying  Gronwall's lemma  we
conclude from the above that
\begin{equation}\label{eq2.6}|X^n_{(k+1)/n}-\bar X^n_{1/n}|\leq e^2\omega_{1/n}(\bar
Y^n,T).\end{equation} Setting $A^n_k=|X^n_{k/n}-\Pi(X^n_{k/n}|$
and using  (\ref{eq2.5}), (\ref{eq2.6}) we have
\begin{eqnarray*}
A^n_{k+1}&\leq&|X^n_{(k+1)/n}-\Pi(\bar X^n_{1/n}|\leq
|X^n_{(k+1)/n}-\bar X^n_{1/n}|+|\bar X^n_{1/n}-\Pi(\bar
X^n_{1/n})|
\\
&\leq&e^2\omega_{1/n}(\bar Y^n,T)+e^{-1}A^n_k.\end{eqnarray*}
Since $A^n_0=0$  and $A^n_1\leq e^2\omega_{1/n}(\bar Y^n,T)$, by
induction on  $k$ we obtain
\begin{equation}\label{eq2.7}
\max_{0\leq
k\leq[nT]}|X^n_{k/n}-\Pi(X^n_{k/n})|\leq\frac{e^2}{1-e^{-1}}\omega_{1/n}(\bar
Y^n,T).\end{equation} Furthermore, for $k=0,1,....,[nT]$ and
$s\in[0,1/n]$ such that $k/n+s\leq T$,
\begin{eqnarray*}
X^n_{k/n+s}-X^n_{k/n}&=&\bar Y^n_{k/n+s}-\bar
Y^n_{k/n}\\
&&-n\int_0^s((X^n_{k/n+u}-X^n_{k/n})-(\Pi(X^n_{k/n+u})-\Pi(X^n_{k/n}))du\\
&&-ns(X^n_{k/n}-\Pi(X^n_{k/n}).\end{eqnarray*} Hence, by
Gronwall's lemma,
\[\sup_{s\in[0,1/n]}|X^n_{(k/n+s)\wedge T}-X^n_{k/n}|\leq
e^2(\omega_{1/n}(\bar Y^n,T)+|X^n_{k/n}-\Pi(X^n_{k/n})|),\] which
when combined with (\ref{eq2.7})  gives
\begin{equation}\label{eq2.8}\max_{0\leq
k\leq[nT]}\sup_{s\in[0,1/n]}|X^n_{(k/n+s)\wedge T}-X^n_{k/n}|\leq
C \omega_{1/n}(\bar Y^n,T)
\end{equation}
for some $C>0$. Of course (\ref{eq2.7}), (\ref{eq2.8}) and
(\ref{eq2.2}) imply (i).

(ii)  By (\ref{eq1.3}), (\ref{eq2.1}) and the well known Aldous
criterion (see e.g. \cite{al}),
 \[\{\bar Y^n\}\quad\mbox{\rm is
tight in}\,\,C({\Bbb R}^+,{\Bbb R}^{d}).\] Moreover,  $\{\bar
Y^n\}$ satisfies the so called UT condition (see e.g.
\cite{s2,s4}) and hence  its every weak limit point is a
semimartingale. Due to  part (i), the sequence $\{ Y^n\}$ is also
tight in $C({\Bbb R}^+,{\Bbb R}^{d})$. Assume that $Y^{(n)}\arrowd
\bar Y$ in $C({\Bbb R}^+,{\Bbb R}^{d})$ along some subsequence. By
\cite[Corollary A3]{s4},
 \[(X^{(n)},K^{(n)})\arrowd (X,K)\quad\mbox{\rm in}\,\,C({\Bbb R}^+,{\Bbb
R}^{2d}),\] where $(\bar X,\bar K)$ is a unique solution of the
Skorkhod problem associated with a semimartingale  $\bar Y$.
\end{proof}
\begin{remark}{\rm
Under (\ref{eq1.3}),
\begin{equation}\label{eq2.9}
\sup_nE|K^n|_T^p<+\infty
\end{equation}
for every $p\geq 1$, $T>0$. This follows from  (\ref{eq2.1}) and
\cite[Theorem
2.5]{laus}. 
}\end{remark}

\nsubsection{Approximations of weak and strong solutions} We  say
that the SDE (\ref{eq1.1}) has a strong solution  if there exists
a pair $(X,\,K)$  of $\{{\cal F}_t\}$\,-\,adapted processes
satisfying (\ref{eq1.1}) and such that $(X,\,K)$ is a solution of
the Skorokhod problem associated with
\begin{equation}\label{eq3.1}
 Y_t=x_0 + \int_0^t \sigma(s,X_s)\,dW_s+\int_0^t
b(s,X_s)\,ds,\quad t\in\Rp.\end{equation} Recall also that the SDE
(\ref{eq1.1}) is said to have a weak solution if  there exists a
probability space $(\bar\Omega,\,\bar{\cal F},\,\bar P)$, an
$\{\bar{\cal F}_t\}$\,-\,adapted Wiener process  $\bar W$ and a
pair of $\{\bar{\cal F}_t\}$\,-\,adapted processes $(\bar X,\bar
K)$ saisfying (\ref{eq1.1}) with $\bar W$ instead of $W$.

 The following
set of general conditions  was introduced in Rozkosz and S\l
omi\'nski \cite{rs1}.
  We
 say that condition  (H) is satisfied if for some closed
subsets $H,\,H_1$ of ${\Bbb R}^+\times\Rd$ such that $H_1\subset
H$,
\begin{itemize}
\item $\forall_{\varepsilon>0}\, \{(\det \sigma_n\sigma^*_n)^{-1}\}_{n\in\Bbb
N}$ is  uniformly integrable on each bounded subset of
$H^c(\varepsilon)$,
\item $\sigma_n\rightarrow
\sigma,\,b_n\rightarrow b$ a.e. on $H^c={\Bbb
R}^+\times\Rd\setminus H$,
\item for every $(t,x)\in
H_1$ (for every $(t,x)\in H$),
\[\sigma_n(t,x_n)\rightarrow
\sigma(t,x),\,\, b_n(t,x_n)\rightarrow b(t,x)\]
 for all $\{x_n\}$
such that $x_n\rightarrow x$ (for all $\{(t,x_n)\}\subset H$ such
that $x_n\rightarrow x$).
\end{itemize}
Here $H^c(\varepsilon)={\Bbb R}^+\times\Rd\setminus
H(\varepsilon)$ and $H(\varepsilon)=H\cup H_{1,\varepsilon}$,
where $H_{1,\varepsilon}=\emptyset$ if $H_1=\emptyset$ and $
H_{1,\varepsilon}=\{z\in{\Bbb R}^+\times\Rd:\inf_{y\in H_1}|z-y|
\le\varepsilon\}$, otherwise.

\begin{theorem} Assume that (\ref{eq1.3}) and (H) are satisfied.
\label{thm2} \begin{enumerate} \item[{\bf(i)}] If the  SDE
(\ref{eq1.1}) has a unique weak solution $X$ then \[X^n\arrowd
X\quad in\,\,C({\Bbb R}^+,{\Bbb R}^d).\]
\item[{\bf (ii)}] If $W^n\arrowp W$ in $C({\Bbb R}^+,{\Bbb R}^d)$ and the  SDE (\ref{eq1.1})
is pathwise unique   then
\[X^n\arrowp X\quad in\,\,C({\Bbb R}^+,{\Bbb R}^d),\]
where $X$ is a unique strong solution of
(\ref{eq1.1}).\end{enumerate}
\end{theorem}
\begin{proof} We  use notations  from the proof of  Theorem
\ref{thm1}.

(i) Our method of proof will be adaptation of  the proof of
\cite[Theorem 2.2]{rs1}. Since  $K^n$ is a bounded variation
process,   one can observe that Krylov's inequality used in
\cite[Theorem 5.1]{rs1} is still in force, i.e.
 there exists a constant $C$ depending only on $d$, $R$ and $t$
 such that for every
non-negative measurable $f:\Rp\times\Rd\to\Rp$,
\begin{equation}\label{eq3.2}
E\int_0^{t\wedge\tau_n^R}f(s,X^n_s)ds \leq C||(\det
\sigma_n\sigma^*_n)^{-1/(d+1)}f||_{{\Bbb L}_{d+1}([0,t]\times
B(0,R))},\end{equation}
 where
$\tau_n^{R}=\inf\{t:|X_t|\vee|K^n|_t>R\}$,
$B(0,R)=\{y\in\Rd:\,|y|<R\}$. By Theorem \ref{thm1}(ii),
$\{(X^n,W^n)\}$ is tight in $C({\Bbb R}^+,{\Bbb R}^{2d})$ and  we
may assume  that $(X^{(n)},W^{(n)})\arrowd(\bar X,\bar W)$ in
$C({\Bbb R}^+,{\Bbb R}^{2d})$ along some subsequence, where  $\bar
W$  is a  Wiener process with respect to the natural filtration
${\cal F}^{\bar X,\bar W}$. By (\ref{eq3.2}) and arguments from
the proof of \cite[Theorem 2.2]{rs1},
\[(X^{(n)},\bar Y^{(n)},W^{(n)})\arrowd(\bar X, \bar Y,\bar W)\quad\mbox{\rm
in}\,\,C({\Bbb R}^+,{\Bbb R}^{3d}),\] where $\bar
Y_t=x_0+\int_0^t\sigma(s,\bar X_s)d\bar W_s+\int_0^tb(s,\bar
X_s)ds$, $t\in\Rp$. Since $Y^n=\bar Y^n-X^n+\Pi(X^n)$, it  follows
from  Theorem \ref{thm1}(ii) 
  that
\[(X^{(n)},K^{(n)},\bar Y^{(n)})\arrowd(\bar X, \bar Y,\bar K)\quad\mbox{\rm
in}\,\,C({\Bbb R}^+,{\Bbb R}^{3d}),\]
 where $(\bar X,\bar K)$ is a solution of the  Skorokhod problem associated with $\bar Y$.
 Hence $(\bar X,\bar K)$ is a weak solution of (\ref{eq1.1}) and the result follows due to weak uniqness of (\ref{eq1.1}).

 (ii) By using arguments from   Gy\"ongy
and Krylov \cite{gk}, to prove that $\{X^n\}$ converges in
probability it is sufficient to show that from any subsequences
$(l)\subset(n), (m)\subset(n) $ it is possible to choose further
subsequences $(l_k)\subset(l),(m_k)\subset(m)$ such that $(
X^{(l_k)}, X^{(m_k)})\arrowd(\bar X,\bar X)$ in $ C({\Bbb
R}^+,{\Bbb R}^{2d}) $, where $\bar X$ is a process with continous
trajectories. From Theorem \ref{thm1}(ii)  we deduce that
\[\{(
X^{(l)},\bar X^{(m)},W^{(l)},W^{(m)})\}\quad\mbox{\rm is tight
in}\,\, C({\Bbb R}^+,{\Bbb R}^{d}).\]
 Therefore, we can choose
subsequences $(l_k)\subset(l),(m_k)\subset(m)$ such that
\[
(X^{(l_k)}, X^{(m_k)},W^{(l_k)},W^{(m_k)})\arrowd (\bar X',\bar
X'',\bar W,\bar W),\quad \mbox{\rm  in}\,\, C({\Bbb R}^+,{\Bbb
R}^{4d}),
\]
 where $\bar X',\bar X''$  are processes with continuous trajectories
and $\bar W$  is a  Wiener process with respect to the natural
filtration ${\cal F}^{\bar X',\bar X'',\bar W}$. In view of  part
(i), the processes $\bar X',\bar X''$  are  solutions of
(\ref{eq1.1}) with $\bar W$ in place of $W$. Since (\ref{eq1.1})
is pathwise unique,  $X'=X''$, and consequently $\{X^n\}$
converges in probability in  $C({\Bbb R}^+,{\Bbb R}^{4d})$ to some
continuous process  $X$. Hence $(X^n,W^n)\arrowp (X,W)$, so using
once again the pathwise uniqueness property of (\ref{eq1.1}) shows
that $X$ is a unique strong solution of (\ref{eq1.1}).
\end{proof}
\begin{remark}{\rm From \cite{rs1} it follows that in fact in  part (i) of the above theorem
the assumption that $\sigma_n\rightarrow \sigma$ a.e. on $H^c$ may
be replaced by a weaker  assumption that
$\sigma_n\sigma^*_n\rightarrow \sigma\sigma^*$  a.e. on $H^c$.}
\end{remark}

\begin{remark}{\rm  There are important examples of
equations of the form (\ref{eq1.1}) with discontinuous
coefficients having unique weak or strong solutions. For instance,
in Schmidt \cite{sz} it is shown that if $d=1$, $D=(r_1,r_2)$,
$b\equiv0$ and $\sigma$ is purely function of $x$, then
(\ref{eq1.1}) has a unique weak solution for every starting point
$x_0\in\bar D$ if and only if the set $M$ of all $x\in\bar D$ such
that $\int_{\bar D\cap U_x}\sigma^{-2}(y)\,dy=+\infty$ for every
open neighborhood $U_x$ of $x$ is equal to the set $N$ of zeros of
$\sigma$. In multidimensional case it is known, that a solution of
(\ref{eq1.1}) is unique in law if (\ref{eq1.3}) is satisfied with
$\sigma_n,b_n$ replaced by $\sigma,b$, the coefficient
$\sigma\sigma^*$ is bounded, continuous and uniformly elliptic,
and $\partial D$ is regular (see Stroock and Varadhan \cite{sw}
for more details). Recently  Semrau \cite{se} considered the
classical case $d=1$, $D=\Rp$ with coefficients $\sigma,b$
depending only on $x$. She has shown  that if $\sigma,b$, satisfy
(\ref{eq1.3}), $\sigma$ is uniformly positive and
$(\sigma(x)-\sigma(y))^2\leq|f(x)-f(y)|$, $x,y\in\Rp$ for some
bounded increasing  function $f:\Rp\to\R$ then there exists a
unique strong solution of (\ref{eq1.1}). Some weaker results on
pathwise uniqueness  can be found in the earlier paper by Zhang
\cite{zh}.

Since in condition (H) we do not require continuity of the limit
coefficients $\sigma$  and $ b$, Theorem \ref{thm2} is a useful
tool for practical approximations of solutions of the equations
mentioned above. }\end{remark}

\nsubsection{Rate of convergence in the case of Lipschitz
continuous coefficients} In this section we  assume that
$\sigma_n=\sigma$, $b_n=b$, $n\in\N$,  where $\sigma, b$ are
Lipschitz continuous functions with respect  to $x$, i.e. satisfy
(\ref{eq1.4}). We also assume  that all SDEs with penalization
term are driven by a fixed $\{{\cal F}_t\}$\,-\,Wiener process
$W$. In particular,  this means that $X^n$ is a solution of the
equation
\begin{equation}\label{eq4.1}
 X^n_t = x_0 + \int_0^t \sigma(s,X^n_s)\,dW_s+\int_0^t
b(s,X^n_s)\,ds -n\int_0^t(X^n_s-\Pi(X^n_s))ds,\quad t\in{\Bbb
R}^+.
\end{equation}
Tanaka \cite{ta} has shown that in the case of Lipschitz continous
coefficients there exists a unique strong solution $(X,K)$ of
(\ref{eq1.1}).  Moreover, from \cite[Theorem 2.2]{s4} and
Gronwall's lemma it follows that
\begin{equation}
\label{eq4.2} E\sup_{t\leq T} |X_t|^p<+\infty\quad\mbox{\rm
and}\quad E|K|_T^p<+\infty
\end{equation}
for every $p\geq1$, $T>0$.
\begin{theorem}\label{thm3} Assume
that (\ref{eq1.3}) and (\ref{eq1.4}) are satisfied. Let $X^n$
satisfy (\ref{eq4.1}), $n\in\N$. For every  $p\in \N$, $T>0$ there
is $C>0$ such that
 \begin{description}
\item[(i)] if  $D$  is a convex polyhedron then
\[
||\sup_{t\leq T}| X^n_t-X_t|||_p\leq C\big(\frac{\ln
n}n\big)^{1/2},\quad n\in\N,
\]
\item[(ii)] if  $D$  is a general convex domain then
\[
||\sup_{t\leq T}| X^n_t-X_t|||_p\leq C\big(\frac{\ln
n}n\big)^{1/4},\quad
 n\in\N,
\]
\end{description}
where  $X$  is a  unique strong solution of
(\ref{eq1.1})\end{theorem}
\begin{proof}  Fix $T>0$. Without loss of generality we may assume that $p\geq2$.

(i)   By   Theorem 2.2 from Dupuis and Ishi \cite{di} there exists
  $C>0$  such that
\begin{eqnarray}
\nonumber&&\sup_{s\leq t}|\Pi(X^n_s)- X_s|\leq C \sup_{s\leq t}|
Y^n_s- Y_s|\\&&\qquad\qquad\label{eq4.3} \leq C\big(\sup_{s\leq
t}|\Pi(X^n_s)- X^n_s|+ \sup_{s\leq t}|\bar Y^n_s- Y_s|\big)
\end{eqnarray} for every $t\leq T$, where $\bar
Y^n_s=x_0+\int_0^s\sigma(u,X^n_u)dW_u+\int_0^sb(u,X^n_u)du$,
$Y^n_s=\bar Y^n_s+\Pi(X^n_s)-X^n_s$, $s\leq T$, $n\in\N$.
Therefore, by Theorem \ref{thm1}(i), Burkholder-Davis-Gundy and
Schwarz's inequalities,
\begin{eqnarray*} & & E\sup_{s\leq t}|X^n_s- X_s|^{p} \leq \const(
(\frac{\ln n}{n})^{p/2}+
E\sup_{s\leq t}|\bar Y^n_s-Y_s|^{p}\nonumber\\
& &\qquad\qquad \leq \const( (\frac{\ln
n}{n})^{p/2}+E\int_0^t||\sigma(s,X^n_s)-\sigma(s,X_s)||^pds\\
&&\qquad\qquad\qquad+E\int_0^t|b(s,X^n_s)-b(s,X_s)|^{p}ds )
\end{eqnarray*}
 for every $t\leq T$. By the above  and  (\ref{eq1.4}), \[E\sup_{s\leq
t}|X^n_s- X_s|^{p}
 \leq \const( (\frac{\ln
n}{n})^{p/2}+\int_0^tE\sup_{u\leq s}|X^n_u-X_u|^{p}ds ) \]for
every $t\leq T$, so (i) follows by Gronwall's lemma.

(ii) If  $D$ is a general convex domain  then by Lemma 2.2 in
Tanaka  \cite{ta},
\begin{eqnarray}
&&|\Pi(X^n_t)- X_t|^2\leq  |\bar Y^n_t-Y_t|^2+ 2\int_0^t (Y^n_t-
Y_t-Y^n_s+Y_s)\,d(K^n_s- K_s)|\nonumber\\
&&\qquad\leq\const\big(|\Pi(X^n_t)- X^n_t|^2+
|\bar Y^n_t- Y_t|^2+\sup_{t\leq T}|\Pi(X^n_t)- X^n_t|(K^n|_T+|K|_T)\nonumber\\
&&\quad\qquad+ \label{eq4.4}|\int_0^t (\bar Y^n_t- Y_t-\bar
Y^n_s+Y_s)\,d(K^n_s- K_s)|
\end{eqnarray}
for every $t\leq T$. Since by the integration by parts formula,
\begin{eqnarray*}
&&\int_0^t (\bar Y^n_t- Y_t-\bar Y^n_s+Y_s)\,d(K^n_s-
K_s)\\
&&\qquad=\int_0^t (X^n_s-X_s)d(\bar Y^n_s-Y_s)+\frac12([\bar
Y^n-Y]_t-|\bar Y^n_t-Y_t|^2)\end{eqnarray*} (here $[\bar Y^n-Y]$
denotes the quadratic variation of $\bar Y^n-Y$), it follows from
Theorem \ref{thm1}(i) and (\ref{eq4.4}) that
\begin{eqnarray*}
&&E\sup_{s\leq t}|X^n_s- X_s|^{p} \leq \const( (\frac{\ln
n}{n})^{p/2}+
E\sup_{s\leq t}|\bar Y^n_s-Y_s|^{p}+E([\bar Y^n-Y]_t)^{p/2}\\
&&\qquad\qquad\qquad\qquad\qquad+E(\sup_{t\leq T}|\Pi(X^n_t)-
X^n_t|)^{p/2}(|K^n|_t+|K|_t)^{p/2}\\
 & &\qquad\qquad\qquad\qquad\qquad+E\sup_{s\leq t}|\int_0^t (X^n_s-X_s)d(\bar
 Y^n_s-Y_s)|^{p/2}\big).
\end{eqnarray*}
By Schwarz's inequality, Theorem \ref{thm1}(i), (\ref{eq2.9}) and
(\ref{eq4.2}),
\[
E(\sup_{t\leq T}|\Pi(X^n_t)-
X^n_t|)^{p/2}(|K^n|_T+|K|_T)^{p/2}\leq \const((\frac{\ln
n}{n})^{p/4}).\] Since $Y^n-Y$ is a continuous semimartingale with
a martingale part
$M^n=\int_0^{\cdot}\sigma(s,X^n_s)-\sigma(s,X_s)dW_s$ and a
bounded variation part $V^n=\int_0^{\cdot}b(s,X^n_s)-b(s,X_s)ds$,
using  Burk\-hol\-der-Davis-Gundy and Schwarz's inequalities we
get
\begin{eqnarray*}
&&E\sup_{s\leq t}|\int_0^t (X^n_s-X_s)d(\bar
 Y^n_s-Y_s)|^{p/2}\leq
 \const\big(E(\int_0^t|X^n_s-X_s|^2d[M^n]_s)^{p/4}\\
 &&\qquad\qquad\qquad\qquad+E(\sup_{s\leq t}|X^n_s-X_s||V^n|_t)^{p/2}\big)\\
 &&\qquad\qquad\qquad\leq \const(E\sup_{s\leq
 t}|X^n_s-X_s|^p)^{1/2}(E([M^n]_t)^{p/2}+(|V^n|_t)^p)^{1/2}.\end{eqnarray*}
Observing that $[\bar Y^n-Y]=[M^n]$  and using the elementary
inequality
 $2ab\leq \epsilon^2 a^2+(b/\epsilon)^2$ with some sufficiently small $\epsilon$ we deduce from the above that
 \begin{eqnarray*}
&&E\sup_{s\leq t}|X^n_s- X_s|^{p} \leq \const\big( (\frac{\ln
n}{n})^{p/4}+ E\sup_{s\leq t}|\bar
Y^n_s-Y_s|^{p}\\
&&\qquad\qquad\qquad\qquad\qquad+E([M^n]_t)^{p/2}+E(|V^n|_t)^p)\big)
\\&&\qquad\qquad\qquad\leq \const\big( (\frac{\ln
n}{n})^{p/4}+E\int_0^t||\sigma(s,X^n_s)-\sigma(s,X_s)||^pds\\
&&\qquad\qquad\qquad\qquad\qquad+E\int_0^t|b(s,X^n_s)-b(s,X_s)|^{p}ds
\big)\\ &&\qquad\qquad\qquad\leq\const\big ( (\frac{\ln
n}{n})^{p/4}+\int_0^tE\sup_{u\leq s}|X^n_u-X_u|^{p}ds \big)
\end{eqnarray*}
for every $t\leq T$. Using Gronwall's lemma  completes the proof.
\end{proof}

\begin{remark}{\rm
In the case of bounded convex domains and bounded Lipschitz
continuous coefficients $\sigma,b$ the problem of $\LL^p$
approximation of solutions of (\ref{eq1.1}) by sequences of
solutions of (\ref{eq4.1}) was considered earlier in Menaldi
\cite{me}. In  particular, in \cite[Theorem 3.1]{me} it is proved
that for every $p\geq1$ and  $T>0$,  $||\sup_{t\leq
T}|X^n_t-X_t|||_p\to0$. From the  proof of  \cite[Theorem 3.1]{me}
 one can also deduce  that
\begin{equation}\label{eq4.5}
\forall_{\delta>0}\,\,\,||\sup_{t\leq T}|X^n_t-X_t|||_p={\cal
O}\big((\frac{1}{n})^{1/4-\delta}\big).\end{equation} In fact, in
\cite[Remark 3.1]{me}  a better rate is stated. However, R.
Pettersson has observed that there is a  gap in the proof of
\cite[Theorem 3.1]{me} (in the first line on  page 741 $p$ should
be replaced by $2p$). Using  Menaldi's calculations  and taking
into account Pettersson's remark one can  only prove
(\ref{eq4.5}). It is also worth pointing out that  Menaldi's
method of proof of (\ref{eq4.5}) is completely different from our
method based on estimates of $\LL^p$-modulus of continuity for
It\^o processes. }
\end{remark}
\mbox{}\\
{\bf Acknowledgements}\\The author is greatly indebted to R.
Pettersson  for stimulating conversation during the conference
"Skorokhod space. 50 years on".


{\small
\begin{thebibliography}{aaa}
\bibitem[1]{a}
 D.J. Aldous,  Stopping time and tightness,
{\em Ann. Probab.} {\bf 6} (1978) 335--340.
\bibitem[2]{al} S. V.
Anulova, R. Sh.  Liptser, Diffusional approximation for processes
with the normal reflection, {\em Theory Probab. Appl.} {\bf 35}
(1990), 411--423.
\bibitem[3]{as}
S. Asmussen, Queueing simulation in heavy traffic, {\em Math.
Oper. Res.} {\bf 17}  (1992), 84--111.
\bibitem[4]{ce}
E. C\'epa,  Probl\`eme de Skorohod multivoque, {\em Ann. Probab. }
{\bf 26} no. 2 (1998), 500--532.
\bibitem[5]{di}
P. Dupuis, H. Ishii, On Lipschitz continuity of the solution
mapping to the  Skorokhod problem with applications, {\em
Stochastics Stochastics Rep.} {\bf  35} (1991), 31--62.
\bibitem[6]{dr}
P. Dupuis, K. Ramanan, A multiclass feedback queueing network with
a regular Skorokhod problem,  {\em Queueing Systems} {\bf 36}
(2000), 327--349.
\bibitem[7]{fn} M. Fischer M, G. Nappo, On the
modulus of continuity of It\^o processes, {\em Stochastic Analysis
and Applications} {\bf 28} (2010), 103--122.
\bibitem[8]{ks}
S. Kanagawa, Y.  Saisho,  Strong approximation of reflecting
Brownian motion using penalty method and its application to
computer simulation, {  \em Monte Carlo Methods and Appl.}, {\bf
6} No. 2 (2000),  105--114.
\bibitem[9]{gk}
I. Gy\"ongy, N. Krylov, Existence of strong solutions for It\^o's
stochastic  equations via approximations, {\em Probab. Theory
Related Fields} {\bf 105} (1996), 143--158.
\bibitem[10]{KS}
P. Kr\'ee, C. Soize, {\em Mathematics and Random Phenomena}
(1986), Dordrecht: Reidel.
\bibitem[11]{ls}
P. L. Lions, A. S.    Sznitman, { Stochastic Differential
Equations with Reflecting Boundary Conditions}, {\em Comm. Pure
and Appl. Math.}  { \bf XXXVII} (1983),  511--537.
\bibitem[12]{lms}
P. L. Lions, J. L.  Menaldi,   A. S. Sznitman, Construction de de
processus de diffusion r\'efl\'echis par p\'enalisation du
domaine, { \em   C.R. Acad. Sci. Paris S\'er. I Math.}
 {\bf 292} (1981),  559--562.
\bibitem[13]{li}
Y. Liu, { Numerical approches to stochastic differential equations
with boundary conditions}, {\em Thesis}, Purdue University, 1993.
\bibitem[14]{laus}
W. \L aukajtys, L.  S\l omi\'nski, Penalization methods for
reflecting stochastic differential equations with jumps, { \em
Stoch. Stoch. Rep.} {\bf 75} no. 5 (2003),  275--293.
\bibitem[15]{lau}
W. \L aukajtys,  {On Stochastic  Differential Equations with
Reflecting Boundary Condition in Convex Domains}, { \em Bull. Pol.
Acad. Sci. Math.} {\bf  52} no. 4 (2004), 445--455.
\bibitem[16]{me}
J. L. Menaldi, Stochastic variational inequality for reflected
diffusion { \em Indiana Univ. Math. J.} {\bf 32} (1983), 733--744.
\bibitem[17]{mr}
J. L. Menaldi, M.  Robin, Reflected diffusion processes with
jumps, {\em Ann. Probab.} {\bf  13} (2) (1985),  319--341.
\bibitem[18]{p1}
R. Pettersson, Approximations for stochastic differential equation
with reflecting convex boundaries, {\em Stochastic Process.
Appl.}, {\bf 59} (1995),  295--308.
\bibitem[19]{p2}
R. Pettersson,  Penalization schemes for reflecting stochastic
differential equations, { \em Bernoulli}, { \bf 3}(4) (1997),
403--414.
\bibitem[20]{rs1}
Rozkosz A., S\l omi\'nski L., On stability and existence of
solutions of SDEs with reflection at the boundary, {\em Stochastic
Process. Appl.} {\bf 68} (1997) 285--302.
\bibitem[21]{sz}
 W. Schmidt,  On stochastic differential equations with
reflecting barriers, {\em Math. Nachr.} {\bf 142} (1989) 135--148.
\bibitem[22]{st}
Y. Saisho and H. Tanaka, On the symmetry of a reflecting Brownian
motion defined by Skorohod's equation for multi--dimensional
domain, {  \em Tokyo J. Math.}  {\bf 10}, No. 2, (1987), 419--435.
\bibitem[23]{se}
A. Semrau, Discrete approximations of strong solutions of
reflecting SDEs with discontinuous  coefficients, {\em Bull. Pol.
Acad. Sci. Math.} {\bf 57} No. 2 (2009), 169--180.
\bibitem[24]{ss}
L. A. Shepp, A. N. Shiryaev, A new look  at pricing of the
"Russian option", {\em Theory Probab. Appl.} {\bf 39} (1994),
103--119.
\bibitem[25]{s2}
L. S\l omi\'nski,  Stability of stochastic differential equations
driven by general semimartingales, {\em Diss. Math.} {\bf CCCXLIX}
(1996), 1--113.
\bibitem[26]{s4}
L. S\l omi\'nski,   {Euler's approximations of solutions of
 SDEs with reflecting boundary},  { \em Stochastic Process. Appl.},
 {\bf
94} (2001),  317--337.
\bibitem[27]{sto}
A. Storm,  Stochastic differential equations with convex
constraint, { \em Stochastics Stochastics Rep.} {\bf 53} (1995),
241--274.
\bibitem[28]{sw}
D.W. Stroock and  S.R.S. Varadhan,  Diffusion Processes with
Boundary Conditions,{\em Comm. Pure Appl. Math.} {\bf 24} (1971)
147--225.
\bibitem[29]{ta}
H. Tanaka, { Stochastic differential equations with reflecting
boundary condition in convex regions},
 {  \em Hiroshima Math. J.}  {\bf 9} (1979),  163--177.
\bibitem[30]{zh}
T. S. Zhang, On strong solutions of one-dimensional stochastic
differential equations with reflecting boundary, { \em Stochastic
Process. Appl.},
 {\bf
50} (1994),  135--147.
\end{thebibliography}
}
\end{document}